\documentclass[11pt]{article}\UseRawInputEncoding
\usepackage{amssymb,amsfonts,amsmath,latexsym,epsf,tikz,url}
\usepackage[usenames,dvipsnames]{pstricks}
\usepackage{pstricks-add}
\usepackage{epsfig}
\usepackage{pst-grad} % For gradients
\usepackage{pst-plot} % For axes
\usepackage[space]{grffile} % For spaces in paths
\usepackage{etoolbox} % For spaces in paths
\makeatletter % For spaces in paths
\patchcmd\Gread@eps{\@inputcheck#1 }{\@inputcheck"#1"\relax}{}{}
\makeatother

\newtheorem{theorem}{Theorem}[section]
\newtheorem{proposition}[theorem]{Proposition}
\newtheorem{conjecture}[theorem]{Conjecture}
\newtheorem{corollary}[theorem]{Corollary}
\newtheorem{lemma}[theorem]{Lemma}
\newtheorem{remark}[theorem]{Remark}

\newcommand{\qed}{\hfill $\square$\medskip}

\begin{document}

\title{On the Sombor characteristic polynomial and Sombor energy of a graph}

\author{
Nima Ghanbari %$^{}$\footnote{Corresponding author}
}

\date{\today}

\maketitle

\begin{center}
	Department of Informatics, University of Bergen, P.O. Box 7803, 5020 Bergen, Norway\\
{\tt  Nima.ghanbari@uib.no }
\end{center}

%%%%%%%%%%%%%%ABSTRACT%%%%%%%%%%%%%%%%%%%%%%%%%%%%%%%%%%%%%%%%%%%%%%%%%%%%%%%%%%%%

\begin{abstract}
Let $G$ be a simple graph with vertex set $V(G) = \{v_1, v_2,\ldots, v_n\}$. The Sombor matrix of $G$, denoted by $A_{SO}(G)$, is defined as the $n\times n$ matrix whose $(i,j)$-entry is $\sqrt{d_i^2+d_j^2}$ if $v_i$ and
$v_j$ are adjacent and $0$ for another cases.
Let the eigenvalues of the Sombor matrix $A_{SO}(G)$ be  $\rho_1\geq \rho_2\geq \ldots\geq \rho_n$ which are the roots of the Sombor characteristic polynomial $\prod_{i=1}^n (\rho-\rho_i)$. The Sombor energy $En_{SO}$ of $G$ is the sum of
absolute values of the eigenvalues of $A_{SO}(G)$. In this paper we compute the Sombor characteristic polynomial and the Sombor energy for  some  graph classes, define Sombor energy unique and propose a conjecture on Sombor Energy.
\end{abstract}


\medskip
\noindent{\bf AMS Subj.\ Class.:} 05C12, 05C50

%%%%%%%%%%%%%%%%%%%%%%%%%%%%%%%%%%%%%%%%%%%%%%%%%%%%%%%%%%%%%%%%%%%%%%%%%%%%%%%%%
%%%%%%%%%%%%%%%%%%%%%%%%%%%%%%%%%%%%%%%%%%%%%%%%%%%%%%%%%%%%%%%%%%%%%%%%%%%%%%%%%
\section{Introduction}

In this paper we are concerned with simple finite graphs, without directed, multiple, or weighted edges, and without self-loops. Let $G=(V,E)$ be such a graph, with vertex set $V(G) = \{v_1, v_2,\ldots, v_n\}$. If two vertices $v_i$ and $v_j$ of $G$ are adjacent, then we use the notation $v_i \sim v_j $. For $v_i \in V(G)$, the degree of the vertex $v_i$, denoted by $d_i$, is the number of the vertices adjacent to $v_i$.

Let $A(G)$ be adjacency matrix of $G$ and $\lambda_1,\lambda_2,\ldots,\lambda_n$ its eigenvalues. These are said to be the eigenvalues of the graph $G$ and to form its spectrum \cite{Cve}. The energy $E(G)$ of the graph $G$ is defined as the sum of the absolute values of its eigenvalues
$$E(G)=\sum_{i=1}^n\vert\lambda_i\vert.$$
Details and more information on graph energy can be found in \cite{Gut,Gut1,Gut2,Maj}. There are many kinds of graph energies, such as Randi\'{c} energy \cite{Alikhani,Boz,Boz1,Das,GutB}, distance energy \cite{Ste}, incidence energy \cite{Boz2},  matching energy \cite{chen,Jis} and Laplacian energy \cite{Das0}. 

Sombor index is defined as  
$SO(G) =\sum_{uv\in E(G)}\sqrt{d_u^2+d_v^2}$ (see \cite{Gutman2}). More details on Sombor index can be found in \cite{Alikhani1,AMC,Symmetry,Deng,Sombor,Red,Wang}. Recently, in \cite{Gutman3}, Gutman introduced sombor matrix of a graph $G$ as $A_{SO}(G)=(r_{ij})_{n\times n}$, and

\begin{displaymath}
 r_{ij}= \left\{ \begin{array}{ll}
\sqrt{d_i^2+d_j^2} & \textrm{if $v_i \sim v_j$}\\
0 & \textrm{otherwise.}
\end{array} \right.
\end{displaymath}

The eigenvalues of $A_{SO}(G)$ are denoted by $\rho_1\geq \rho_2\geq \ldots\geq \rho_n$, and are said to form the Sombor spectrum of the graph $G$. In \cite{Gutman3}, Gutman introduced sombor characteristic polynomial $\phi _{SO}(G,\lambda)$ as

$$\phi _{SO}(G,\lambda)=det(\lambda I -A_{SO}(G))=\prod_{i=1}^n (\lambda-\rho_i),$$

and Sombor energy $En_{SO}(G)$ as

$$En_{SO}(G)=\sum_{i=1}^{n}|\rho_i|.$$

Two graphs $G$ and $H$ are said to be {\it Sombor energy equivalent},
or simply ${\cal EN_{SO}}$-equivalent, written $G\sim H$, if
$En_{SO}(G)=En_{SO}(H)$. It is evident that the relation $\sim$ of being
${\cal EN_{SO}}$-equivalence
 is an equivalence relation on the family ${\cal G}$ of graphs, and thus ${\cal G}$ is partitioned into equivalence classes,
called the {\it ${\cal EN_{SO}}$-equivalent}. Given $G\in {\cal G}$, let
\[
[G]=\{H\in {\cal G}:H\sim G\}.
\]
We call $[G]$ the equivalence class determined by $G$.
A graph $G$ is said to be {\it Sombor energy unique}, or simply {\it ${\cal EN_{SO}}$-unique}, if $[G]=\{G\}$.

 A graph $G$ is called {\it $k$-regular} if all
vertices  have the same degree $k$.  One of the famous graphs is the Petersen
graph which is a symmetric non-planar 3-regular graph.  In the study of Sombor energy, it is interesting to investigate
 the Sombor  characteristic polynomial and Sombor energy of this graph. We denote the Petersen graph by $P$.

In this paper, we consider the Sombor characteristic polynomial and Sombor energy of graphs. In Section 2, we bring some known results about Sombor characteristic polynomial and Sombor energy. Also the Sombor characteristic polynomial and Sombor energy of some special kind of graphs are computed. In Section 3, We Consider to regular graphs, especially cubic graphs of order 10 and state a conjecture. In Section 4, we state some open problems for future direction of this research .

\section{Sombor energy of specific graphs}\label{sec1}
In this section we study the Sombor characteristic polynomial and the Sombor energy for certain graphs. The following result gives McClelland-type bound for the Sombor energy.

\begin{theorem}\cite{Gutman3}
If $G$ is a graph on $n$ vertices, and $F(G)$ is its forgotten topological index, then
$$En_{SO}(G)\leq \sqrt{2nF(G)}.$$
\end{theorem}

The following result gives Koolen-Moulton-type bound for the Sombor energy.

\begin{theorem}\cite{Gutman3}
Let $G$ be a graph on $n$ vertices, with Sombor and forgotten topological indices $SO(G)$ and  $F(G)$, respectively. Then

$$En_{SO}(G)\leq \frac{2SO(G)}{n}+\sqrt{(n-1)\left(2F(G)-\left(\frac{2SO(G)}{n}\right)^2\right)}.$$
\end{theorem}

Here we shall compute the Sombor characteristic polynomial of paths and cycles.

\begin{theorem}\label{thm-path scp}
For every $n\geq 5$, the Sombor characteristic polynomial of the path graph $P_n$ satisfy:
$$\phi _{SO}(P_n,\lambda)=  \lambda ^2 \Lambda _{n-2}-10\lambda \Lambda _{n-3}+25\Lambda _{n-4},$$
where for every $k\geq 3$, $\Lambda _k = \lambda \Lambda _{k-1}-8\Lambda _{k-2}$ with $\Lambda _1=\lambda $ and $\Lambda _2= \lambda ^2 -8$. Also the characteristic polynomial of $P_2,P_3$ and $P_4$ are $\lambda ^2 -2,\lambda ^3 -10 \lambda$ and $\lambda ^4 -18 \lambda ^2 +25$ respectively.
\end{theorem}

\begin{proof}
	It is easy to see that the characteristic polynomial of $P_2$ is $\lambda ^2 -2$, Also for  $P_3$ is $\lambda ^3 -10 \lambda$ and for $P_4$ is $\lambda ^4 -18 \lambda ^2 +25$. Now for every $k\geq 3$ consider

	\[
    	M_k:=\left(
         	\begin{array}{cccccc}
         	\lambda & -\sqrt{8}  & 0  & \ldots  & 0 &   0         \\
			-\sqrt{8} & \lambda & -\sqrt{8}  & \ldots  &  0&   0          \\
         	0& -\sqrt{8} & \lambda  & \ldots  & 0 &   0          \\
         \vdots	& \vdots & \vdots  & \ddots  & \vdots &  \vdots           \\					
			0&  0& 0  & \ldots     & \lambda &   -\sqrt{8}          \\
			0& 0 & 0&  \ldots    & -\sqrt{8} & \lambda            
        	\end{array}
    	\right)_{k\times k},
	\]
	
	and let $\Lambda _k= det(M_k)$. One can easily check that $\Lambda _k = \lambda \Lambda _{k-1}-8\Lambda _{k-2}$ . Now consider the path graph $P_n$. Suppose that $\phi _{SO}(P_n,\lambda)= det (\lambda I - A_{SO}(P_n))$. We have 
	
	\[
    	\phi _{SO}(P_n,\lambda)= det\left(
         	\begin{array}{c|cccccc|c}
         	\lambda & -\sqrt{5}  & 0 & 0 & \ldots & 0 & 0 &   0         \\
         	\hline
			-\sqrt{5} & & & & & & &   0          \\
         	0&  & & & & &  &  0           \\
			0&   & &  & & &  &  0        \\
         \vdots	&  & &  & M_{n-2} & & &  \vdots           \\
			0&   & & & & &   &    0        \\						
			0&  & & & & &  &   -\sqrt{5}          \\
			\hline
			0& 0 & 0& 0 & \ldots   & 0 & -\sqrt{5} & \lambda           
        	\end{array}
    	\right)_{n\times n}.
	\]
	
	So,

\begin{align*}
    	\phi _{SO}(P_n,\lambda)&= 	
    	 \lambda det\left(
         	\begin{array}{cccc|c}
			   & & &  &  0        \\
            &  & M_{n-2} & &  \vdots           \\
			    & & &   &    0        \\						
			 & & &  &   -\sqrt{5}          \\
			\hline
			 0 & \ldots   & 0 & -\sqrt{5} & \lambda           
        	\end{array}
    	\right)\\
    	& \quad+
    	\sqrt{5} det\left(
         	\begin{array}{c|ccc|c}
			 - \sqrt{5} &  - \sqrt{8}  & \ldots  & 0 &  0        \\
			 \hline
          0  &  &  & &  0           \\
			  \vdots  & &M_{n-3} &   &    \vdots       \\						
			0 & & &  &   -\sqrt{5}          \\
			\hline
			 0 & 0  & \ldots  & -\sqrt{5} & \lambda           
        	\end{array}
        	\right).        	
\end{align*}

And so,

\begin{align*}
    	\phi_{SO}(P_n,\lambda)&= 	
    	 \lambda\left(\lambda \Lambda _{n-2} + \sqrt{5} det\left(
         	\begin{array}{cccc|c}
			   & & &  &  0        \\
            &  & M_{n-3} & &  \vdots           \\
			    & & &   &    0        \\						
			 & & &  &   0         \\
			\hline
			 0 & \ldots   & 0 & -\sqrt{8} & -\sqrt{5}           
        	\end{array}
    	\right)\right)\\
    	& \quad -5
    	det\left(
         	\begin{array}{cccc|c}
			   & & &  &  0        \\
            &  & M_{n-3} & &  \vdots           \\
			    & & &   &    0        \\						
			 & & &  &   -\sqrt{5}          \\
			\hline
			 0 & \ldots   & 0 & -\sqrt{5} & \lambda           
        	\end{array}
    	\right)     	
\end{align*}

Hence,

\begin{align*}
    	\phi_{SO}(P_n,\lambda)&= 	
    	 \lambda\left(\lambda \Lambda _{n-2} -5 \Lambda _{n-3}\right)\\
    	& \quad -5 \left(
    	\lambda \Lambda _{n-3} +\sqrt{5} det\left(
         	\begin{array}{cccc|c}
			   & & &  &  0        \\
            &  & M_{n-4} & &  \vdots           \\
			    & & &   &    0        \\						
			 & & &  &   0         \\
			\hline
			 0 & \ldots   & 0 & -\sqrt{8} & -\sqrt{5}           
        	\end{array}
    	\right)\right) \\
    	&= 	
    	 \lambda\left(\lambda \Lambda _{n-2} -5 \Lambda _{n-3}\right)      -5 \left(
    	\lambda \Lambda _{n-3} -5\Lambda _{n-4} \right),	   	
\end{align*}

and therefore we have the result.\qed
\end{proof}

\begin{theorem}\label{thm-cycle}
For every $n\geq 3$, the Sombor characteristic polynomial of the cycle graph $C_n$ satisfy:
$$\phi _{SO}(C_n,\lambda)= \lambda \Lambda _{n-1} -16\Lambda _{n-2} -2(\sqrt{8})^n ,$$
where for every $k\geq 3$, $\Lambda _k = \lambda \Lambda _{k-1}-8\Lambda _{k-2}$ with $\Lambda _1=\lambda $ and $\Lambda _2= \lambda ^2 -8$.
\end{theorem}

\begin{proof}
Similar to the proof of Theorem \ref{thm-path scp}, for every $k\geq3$, we consider

	\[
    	M_k:=\left(
         	\begin{array}{cccccccc}
         	\lambda & -\sqrt{8}  & 0 & 0 & \ldots & 0 & 0 &   0         \\
			-\sqrt{8} & \lambda & -\sqrt{8} & 0 & \ldots & 0 &  0&   0          \\
         	0& -\sqrt{8} & \lambda & -\sqrt{8} & \ldots & 0 & 0 &   0          \\
			0& 0 & -\sqrt{8} & \lambda &  \ldots &0  & 0  &  0        \\
         \vdots	& \vdots & \vdots & \vdots & \ddots & \vdots & \vdots &  \vdots           \\
			0&  0&0 & 0& \ldots   & \lambda & -\sqrt{8}  &    0        \\						
			0&  0& 0 &0 & \ldots   & -\sqrt{8} & \lambda &   -\sqrt{8}          \\
			0& 0 & 0& 0 & \ldots   & 0 & -\sqrt{8} & \lambda            
        	\end{array}
    	\right)_{k\times k},
	\]
	
	and let $\Lambda _k= det(M_k)$. We have $\Lambda _k = \lambda \Lambda _{k-1}-8\Lambda _{k-2}$. Suppose that $\phi _{SO}(C_n,\lambda)= det (\lambda I - A_{SO}(C_n))$. We have 
	
	\[
    	\phi _{SO}(C_n,\lambda)= det\left(
         	\begin{array}{c|ccccccc}
         	\lambda & -\sqrt{8}  & 0 & 0 & \ldots & 0 & 0 &   -\sqrt{8}         \\
         	\hline
			-\sqrt{8} & & & & & & &             \\
         	0&  & & & & &  &             \\
			0&   & &  & & &  &          \\
         \vdots	&  & &  & M_{n-1} & & &             \\
			0&   & & & & &   &          \\						
			0&  & & & & &  &            \\
			-\sqrt{8}&  & &  &   &  &  &           
        	\end{array}
    	\right)_{n\times n}.
	\]
	
	So,

\begin{align*}
    	\phi _{SO}(P_n,\lambda)&= 	\lambda \Lambda _{n-1}+
    	 \sqrt{8} det\left(
         	\begin{array}{c|cccc}
			-\sqrt{8}   & -\sqrt{8} & 0 & \ldots &  0        \\
			   \hline
           0 &  &  & &             \\
			   \vdots & & M_{n-2} &   &            \\						
			0 & & &  &            \\
			 -\sqrt{8} &   &  &  &            
        	\end{array}
    	\right)\\
    	& \quad+
    	(-1)^{n+1}(-\sqrt{8}) det\left(
         	\begin{array}{c|cccc}
			 - \sqrt{8} &    &   &  &         \\
          0  &  &  & M_{n-2} &             \\
			  \vdots  & & &   &        \\						
			0 & & &  &             \\
			\hline
			 -\sqrt{8}   & 0  & \ldots  & 0 &   -\sqrt{8}       
        	\end{array}
        	\right).        	
\end{align*}

Hence,

\begin{align*}
    	\phi_{SO}(P_n,\lambda)&= 	\lambda \Lambda _{n-1}+
    	 \sqrt{8}\left( -\sqrt{8}\Lambda _{n-2} + (-1)^{n}(-\sqrt{8})^{n-1}
    	 \right)\\
    	& \quad +(-1)^{n+1}(-\sqrt{8})
    	\left( (-\sqrt{8})^{n-1}+(-1)^n(-\sqrt{8})\Lambda _{n-2}
    	\right),     	
\end{align*}

and therefore we have the result.\qed
\end{proof}

Now we consider to star graph $S_n$ and find its Sombor characteristic polynomial and Sombor energy. We need the following Lemma.

\begin{lemma} \label{new}\rm\cite{Cve}
If $M$ is a nonsingular square matrix, then
$$det\left(  \begin{array}{cc}
M&N  \\
P& Q \\
\end{array}\right)=det (M) det( Q-PM^{-1}N).
$$
\end{lemma}

\begin{theorem}
For $n\geq 2$,
\begin{itemize}
\item[(i)] The Sombor characteristic polynomial of the star graph $S_n=K_{1,n-1}$ is
$$\phi_{SO}(S_n,\lambda))=\lambda^{n-2}\left(\lambda ^2 -(n-1)(n^2-2n+2)\right).$$
\item[(ii)] The Sombor energy of $S_n$ is
$$En_{SO}(S_n)=2\sqrt{(n-1)(n^2-2n+2)}.$$
\end{itemize}
\end{theorem}

\begin{proof}
\begin{enumerate}
\item[(i)] One can easily check that the Sombor matrix of $K_{1,n-1}$ is
$$A_{SO}(S_n)=\sqrt{n^2-2n+2}\left( \begin{array}{cc}
0_{1\times 1}&J_{1\times {n-1}} \\
J_{{n-1}\times 1}&0_{{n-1}\times {n-1}}   \\
\end{array} \right).$$
We  have 
$$det(\lambda I - A_{SO}(S_n))=det
\left(  \begin{array}{cc}
\lambda  & -\sqrt{n^2-2n+2}J_{1\times (n-1)}  \\
-\sqrt{n^2-2n+2}J_{(n-1)\times 1}& \lambda I_{n-1} \\
\end{array}\right).
$$
Using Lemma \ref{new},
$$
det(\lambda I -A_{SO}(S_n))=\lambda det( \lambda I_{n-1} - \sqrt{n^2-2n+2}J_{(n-1)\times 1}\frac{1}{\lambda} \sqrt{n^2-2n+2}J_{1\times (n-1)}).
$$

Since $J_{(n-1)\times 1}J_{1\times (n-1)}=J_{n-1}$, therefore
\begin{align*}
det(\lambda I - A_{SO}(S_n))&=\lambda det( \lambda I_{n-1} - \frac{1}{\lambda}(n^2-2n+2)J_{n-1})\\
&=\lambda ^{2-n}
det( \lambda ^2 I_{n-1} - (n^2-2n+2)J_{n-1}).
\end{align*}
On the other hand, the eigenvalues of $J_{n-1}$ are $n-1$ (once) and 0 ($n-2$ times),  the eigenvalues of $(n^2-2n+2)J_{n-1}$ are $(n-1)(n^2-2n+2)$ (once) and 0 ($n-2$ times). Hence
$$\phi_{SO}(S_n,\lambda))=\lambda^{n-2}\left(\lambda ^2 -(n-1)(n^2-2n+2)\right).$$

 \item[(ii)] It follows from Part (i).\quad\qed

 \end{enumerate}
\end{proof}

Here we shall investigate the Sombor energy of complete graphs.

\begin{theorem}
For $n\geq 2$,
\begin{itemize}
\item[(i)] The Sombor characteristic polynomial of complete graph $K_n$ is
$$\phi_{SO}(K_n,\lambda)=(\lambda -(n-1)^2\sqrt{2})(\lambda +(n-1)\sqrt{2})^{n-1}.$$
\item[(ii)] The Sombor energy of $K_n$ is
$$En_{SO}(K_n)=2(n-1)^2\sqrt{2}.$$
\end{itemize}
\end{theorem}

\begin{proof}
\begin{enumerate}
\item[(i)]
The Sombor matrix of $K_n$ is $(n-1)\sqrt{2}(J-I)$. Therefore
$$
\phi_{SO}(K_n,\lambda)=det(\lambda I- (n-1)\sqrt{2}J+ (n-1)\sqrt{2}I)=det((\lambda +(n-1)\sqrt{2})I- (n-1)\sqrt{2}J).
$$

Since the eigenvalues of $J_n$ are $n$ (once) and 0 ($n-1$ times),  the eigenvalues of $(n-1)\sqrt{2}J_n$ are $n(n-1)\sqrt{2}$ (once) and 0 ($n-1$ times). Therefore
$$\phi_{SO}(K_n,\lambda)=(\lambda -(n-1)^2\sqrt{2})(\lambda +(n-1)\sqrt{2})^{n-1}.$$

\item[(ii)] It follows from Part (i).\quad\qed
\end{enumerate}
\end{proof}

We end this section by finding Sombor characteristic polynomial of complete bipartite graphs and their Sombor energy.

\begin{theorem}\label{thm-bipartite}
For natural number $m,n\neq 1$,
\begin{itemize}
\item[(i)] The Randi\'{c} characteristic polynomial of complete bipartite graph $K_{m,n}$ is
$$\phi_{SO}(K_{m,n},\lambda)=\lambda^{m+n-2}(\lambda^2 -mn(m^2+n^2)).$$
\item[(ii)] The Randi\'{c} energy of $K_{m,n}$ is
$2\sqrt{mn(m^2+n^2)}$.
\end{itemize}
\end{theorem}

\begin{proof}
\begin{enumerate}
\item[(i)]
It is easy to see that the Sombor matrix of $K_{m,n}$ is
$\sqrt{m^2+n^2}\left( \begin{array}{cc}
0_{m\times m}&J_{m\times n} \\
J_{n\times m}&0_{n\times n}   \\
\end{array} \right)$.
Using Lemma \ref{new} we have 

$$det(\lambda I -A_{SO}(K_{m,n}))=det
\left(  \begin{array}{cc}
\lambda I_m & -\sqrt{m^2+n^2}J_{m\times n}  \\
-\sqrt{m^2+n^2}J_{n\times m}& \lambda I_n \\
\end{array}\right).
$$
So
$$
det(\lambda I -A_{SO}(K_{m,n}))=det (\lambda I_m) det( \lambda I_n - \sqrt{m^2+n^2}J_{n\times m}\frac{1}{\lambda}I_m \sqrt{m^2+n^2}J_{m\times n}).
$$

We know that $J_{n\times m}J_{m\times n}=mJ_n$. Therefore
\begin{align*}
det(\lambda I -A_{SO}(K_{m,n}))&=\lambda ^m det( \lambda I_n - \frac{1}{\lambda }m(m^2+n^2)J_n)\\
&=\lambda ^{m-n}
det( \lambda ^2 I_n - m(m^2+n^2)J_n).
\end{align*}
The eigenvalues of $J_n$ are $n$ (once) and 0 ($n-1$ times). So the eigenvalues of $m(m^2+n^2)J_n$ are $mn(m^2+n^2)$ (once) and 0 ($n-1$ times).
Hence
$$\phi_{SO}(K_{m,n},\lambda)=\lambda^{m+n-2}(\lambda^2 -mn(m^2+n^2)).$$

\item[(ii)] It follows from Part (i).\quad\qed
\end{enumerate}

\end{proof}

\section{Sombor energy of 2-regular and 3-regular graphs}

In this section  we consider $2$-regular and $3$-regular graphs. As a beginning of this section,  we have the following easy lemma:

\begin{lemma}\label{union}
Let $G=G_1\cup G_2\cup G_3\cup \ldots\cup G_n$. Then
\begin{enumerate}
\item[(i)] $\phi_{SO}(G)=\prod_{i=1}^{n}\phi_{SO}(G_i)$.

\item[(ii)] $En_{SO}(G)=\sum_{i=1}^{n}En_{SO}(G_i)$.
\end{enumerate}
\end{lemma}

As an immediate result of Lemma \ref{union}, we have the  following results:

\begin{proposition}
\begin{enumerate}
\item[(i)] If $e=v_rv_{r+1}\in E(P_n)$, then $ En_{SO}(P_n-e)= En_{SO}(P_r)+ En_{SO}(P_s)$, where $r+s=n$.

\item[(ii)] If $e\in E(C_n)$, ($n\geq 3$), then
$En_{SO}(C_n-e)=En_{SO}(P_n).$

\item[(iii)] Let $S_n$ be the star on $n$ vertices and $e\in E(S_n)$. Then for any $n\geq 3$,
$$En_{SO}(S_n-e)=En_{SO}(S_{n-1}).$$
\end{enumerate}
\end{proposition}

Now consider to the $2$-regulars. Every $2$-regular graph is a disjoint union of cycles. By Theorem \ref{thm-cycle}, we can find all the eigenvalues of Sombor matrix of cycle graphs. Therefore by Lemma \ref{union}, we can find Sombor characteristic polynomial and Sombor energy of 2-regular graphs. Now, we consider to  the characteristic polynomial of $3$-regular graphs of order $10$. Also we shall compute sombor energy  of this class of graphs. There are exactly $21$ cubic graphs of  order $10$ given in
Figure \ref{cubic} (see \cite{reza}).

\newpage

\bigskip
\bigskip

Using Maple we computed the Sombor characteristic polynomials of $3$-regular  graphs of order $10$ in Table 1.

\begin{figure}[!h]
\hglue1.2cm
\includegraphics[width=12cm,height=5.3cm]{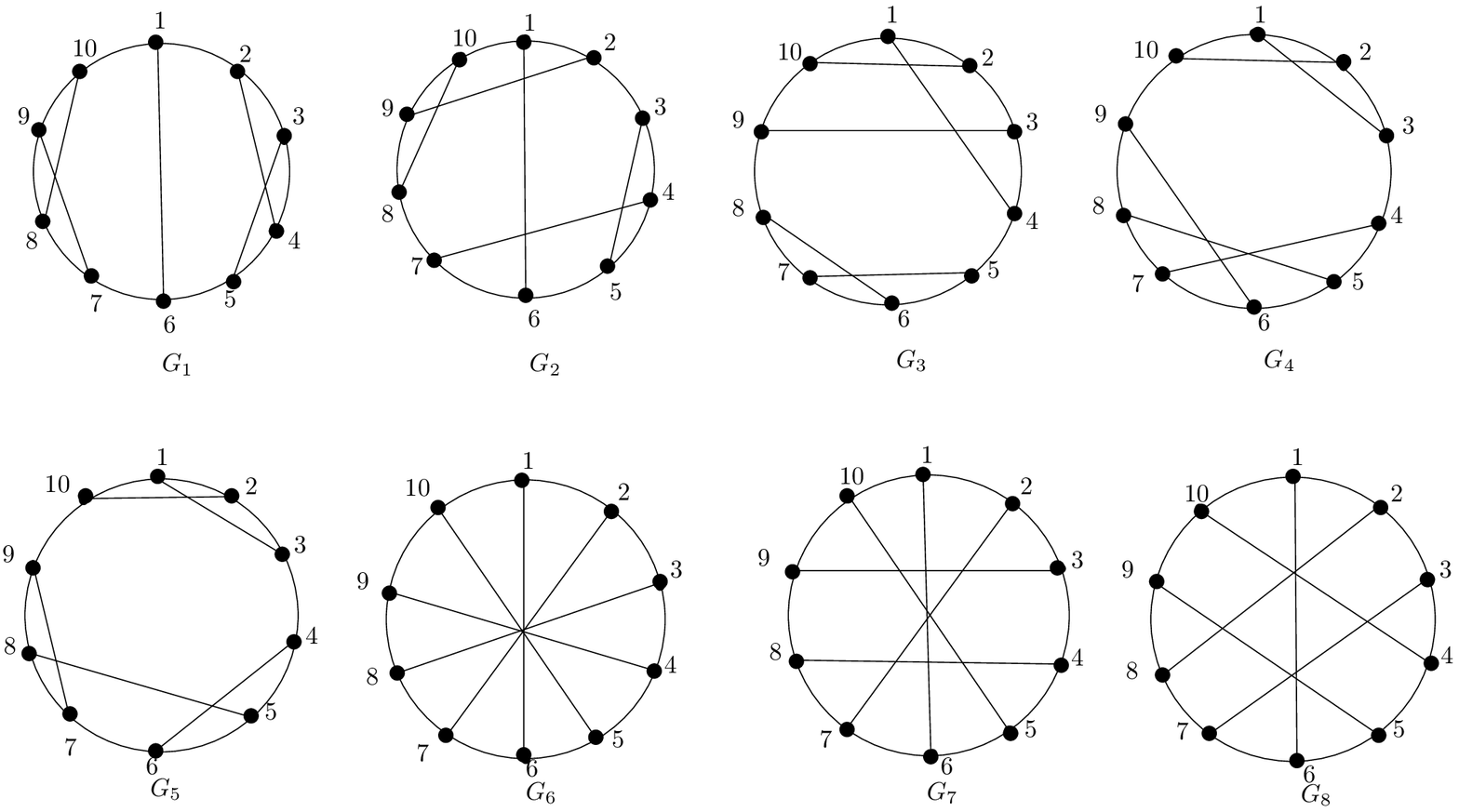}
\vglue5pt
\hglue1.2cm
\includegraphics[width=12cm,height=5.3cm]{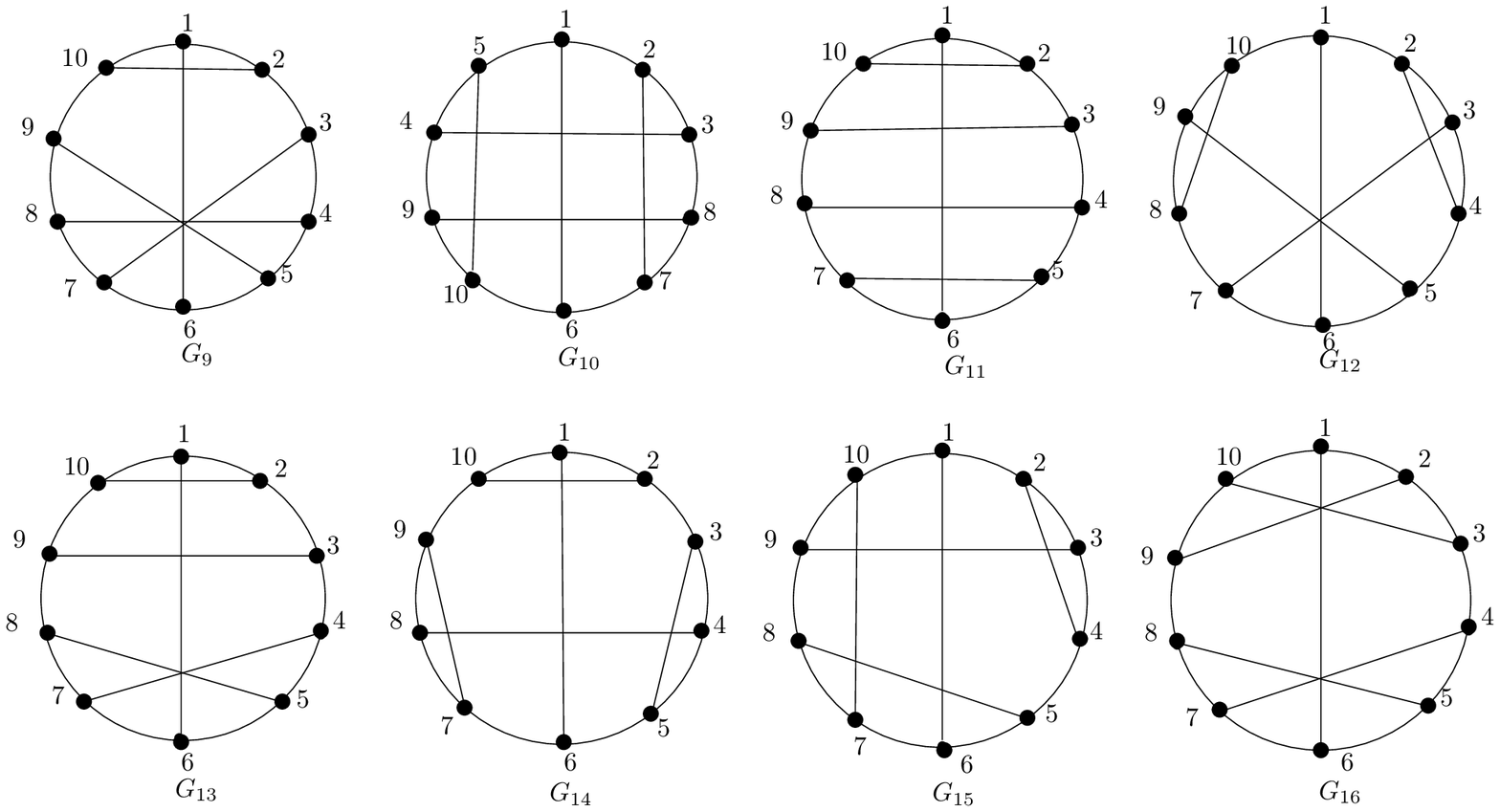}
\hglue1.2cm
\includegraphics[width=11.7cm,height=5.2cm]{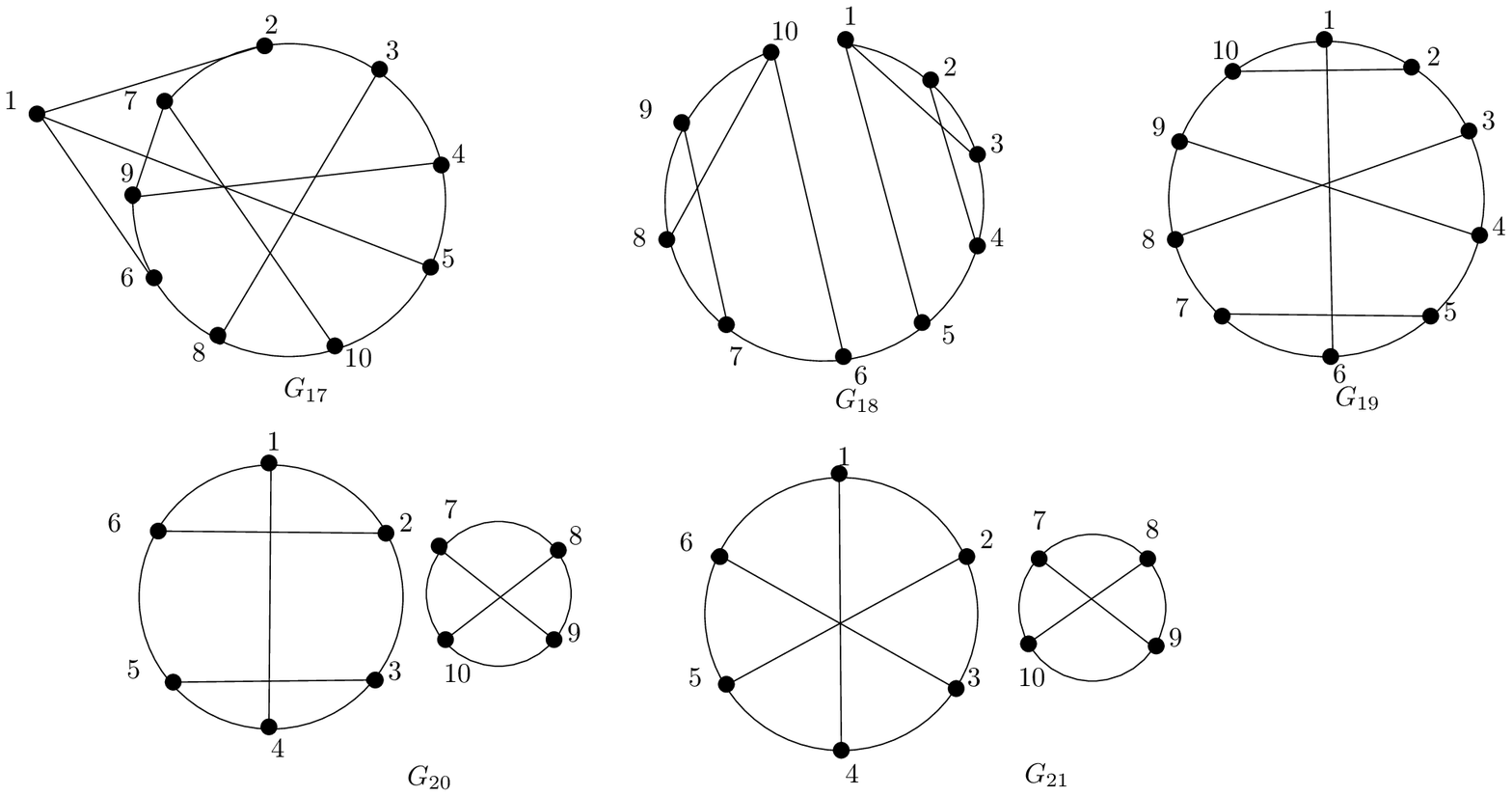}
\vglue-10pt \caption{\label{cubic} Cubic graphs of order 10.}
\end{figure}

 	By finding  the roots of Sombor characteristic polynomial of cubic  graphs of order $10$, we can have the Sombor energy of these graphs.  We compute them to three decimal places. So we have them in table 2.

\begin{center}
\begin{footnotesize}
\begin{tabular} {|c|c|}
\hline
$G_i$ &  $P(G_i,\lambda)$ \\
 \hline $G_1$  & $\lambda ^{10} - 270\lambda^8 - 432\sqrt{2}\lambda^7 +23004\lambda^6+62208\sqrt{2} \lambda ^5 -589032\lambda^4 -1819584\sqrt{2} \lambda^3 +4618944\lambda^2 $ \\ & $   + 15116544\sqrt{2}\lambda$  \\
 \hline $G_2$ & $\lambda^{10} - 270\lambda^8 - 216\sqrt{2}\lambda^7 + 23004\lambda^6 + 27216\sqrt{2}\lambda^5 - 705672\lambda^4 - 839808\sqrt{2}\lambda^3 + 6718464\lambda^2$ \\
 & $ + 7558272\sqrt{2}\lambda$\\
 \hline $G_3$ & $\lambda^{10} - 270\lambda^8 - 324\sqrt{2}\lambda^7 + 22356\lambda^6 + 46656\sqrt{2}\lambda^5 - 559872\lambda^4 - 1329696\sqrt{2}\lambda^3 + 3149280\lambda^2 $ \\
 & $+ 8188128\sqrt{2}\lambda + 5668704$ \\
  \hline $G_4$ & $\lambda^{10} - 270\lambda^8 - 216\sqrt{2}\lambda^7 + 20412\lambda^6 + 34992\sqrt{2}\lambda^5 - 355752\lambda^4 - 979776\sqrt{2}\lambda^3 - 1259712\lambda^2$  \\
 \hline $G_5$  & $\lambda^{10} - 270\lambda^8 - 432\sqrt{2}\lambda^7 + 23004\lambda^6 + 66096\sqrt{2}\lambda^5 - 542376\lambda^4 - 2309472\sqrt{2}\lambda^3 - 3779136\lambda^2$ \\
  \hline $G_6$ & $\lambda^{10} - 270\lambda^8 + 21060\lambda^6 - 612360\lambda^4 + 5773680\lambda^2 - 17006112$ \\
  \hline $G_7$ &  $\lambda^{10} -270\lambda^8 + 22356\lambda^6 - 11664\sqrt{2}\lambda^5 - 682344\lambda^4 + 629856\sqrt{2}\lambda^3 + 6193584\lambda^2 - 3779136\sqrt{2}\lambda$\\
  & $- 17006112$ \\
  \hline $G_8$ & $\lambda^{10}  - 270\lambda^8 + 23004\lambda^6 - 15552\sqrt{2}\lambda^5 - 775656\lambda^4 + 1119744\sqrt{2}\lambda^3 + 7978176\lambda^2 - 15116544\sqrt{2}\lambda$ \\
  \hline $G_9$ & $\lambda^{10}   - 270\lambda^8 - 108\sqrt{2}\lambda^7 + 23004\lambda^6 + 7776\sqrt{2}\lambda^5 - 769824\lambda^4 - 34992\sqrt{2}\lambda^3 + 9552816\lambda^2$ \\
  & $- 2519424\sqrt{2}\lambda - 22674816$ \\
  \hline $G_{10}$ & $\lambda^{10}   - 270\lambda^8 + 21060\lambda^6 - 3888\sqrt{2}\lambda^5 - 495720\lambda^4 - 349920\sqrt{2}\lambda^3 + 3674160\lambda^2 + 6298560\sqrt{2}\lambda$ \\
  & $+ 5668704$ \\
  \hline $G_{11}$ &  $\lambda^{10}  -270\lambda^8 - 216\sqrt{2}\lambda^7 + 22356\lambda^6 + 31104\sqrt{2}\lambda^5 - 612360\lambda^4 - 1119744\sqrt{2}\lambda^3 + 2414448\lambda^2 $\\
  & $+ 6298560\sqrt{2}\lambda + 5668704$\\
  \hline $G_{12}$ & $\lambda^{10}   - 270\lambda^8 - 216\sqrt{2}\lambda^7 + 24300\lambda^6 + 23328\sqrt{2}\lambda^5 - 915624\lambda^4 - 629856\sqrt{2}\lambda^3 + 15116544\lambda^2$ \\
  & $+ 5038848\sqrt{2}\lambda - 90699264$\\
  \hline $G_{13}$ & $\lambda^{10}   - 270\lambda^8 - 108\sqrt{2}\lambda^7 + 21708\lambda^6 + 11664\sqrt{2}\lambda^5 - 559872\lambda^4 - 384912\sqrt{2}\lambda^3 + 3674160\lambda^2$ \\
  & $+ 3779136\sqrt{2}\lambda$\\
  \hline $G_{14}$ &  $\lambda^{10}   -270\lambda^8 - 324\sqrt{2}\lambda^7 + 24300\lambda^6 + 46656\sqrt{2}\lambda^5 - 839808\lambda^4 - 1994544\sqrt{2}\lambda^3 + 7873200\lambda^2$\\
  & $+ 21415104\sqrt{2}\lambda + 22674816$\\
  \hline $G_{15}$ & $\lambda^{10}   -270\lambda^8 - 108\sqrt{2}\lambda^7 + 22356\lambda^6 + 11664\sqrt{2}\lambda^5 - 676512\lambda^4 - 419904\sqrt{2}\lambda^3 + 5668704\lambda^2$ \\
  & $+ 8188128\sqrt{2}\lambda + 5668704$\\
  \hline $G_{16}$ & $\lambda^{10}   -270\lambda^8 + 20412\lambda^6 - 495720\lambda^4 + 3779136\lambda^2$ \\
  \hline $G_{17}$ & $\lambda^{10}  -270\lambda^8 + 24300\lambda^6 - 23328\sqrt{2}\lambda^5 - 962280\lambda^4 + 2099520\sqrt{2}\lambda^3 + 12597120\lambda^2 - 50388480\sqrt{2}\lambda$ \\
  & $+ 9069926$\\
  \hline $G_{18}$ & $\lambda^{10}  -270\lambda^8 - 432\sqrt{2}\lambda^7 + 20412\lambda^6 + 62208\sqrt{2}\lambda^5 - 215784\lambda^4 - 979776\sqrt{2}\lambda^3 - 1259712\lambda^2$ \\
  \hline $G_{19}$ & $\lambda^{10}  -270\lambda^8 - 216\sqrt{2}\lambda^7 + 23652\lambda^6 + 27216\sqrt{2}\lambda^5 - 822312\lambda^4 - 909792\sqrt{2}\lambda^3 + 10392624\lambda^2$ \\
  & $+ 5038848\sqrt{2}\lambda - 39680928$\\
  \hline $G_{20}$ & $\lambda^{10}  -270\lambda^8 - 648\sqrt{2}\lambda^7 + 20412\lambda^6 + 93312\sqrt{2}\lambda^5 - 75816\lambda^4 - 1469664\sqrt{2}\lambda^3 - 3779136\lambda^2$ \\
  \hline $G_{21}$ & $\lambda^{10}  -270\lambda^8 - 432\sqrt{2}\lambda^7 + 16524\lambda^6 + 69984\sqrt{2}\lambda^5 + 157464\lambda^4$ \\ \hline
  \end{tabular}
\end{footnotesize}
\end{center}
\begin{center}
{Table 1.} Sombor characteristic polynomial $P(G_i,\lambda)$, for $1\leq i \leq 21$.
\end{center}

\begin{center}
\begin{footnotesize}
\small
\begin{tabular}{|c|c|||c|c|||c|c|} \hline
$G_i$ & $En_{SO}(G_i)$ &$G_i$ & $En_{SO}(G_i)$ &$G_i$ & $En_{SO}(G_i)$ \\
\hline
\hline $G_1$ & 64.161 &  $G_8$ & 64.161  & $G_{15}$ & 62.767  \\
\hline $G_2$ & 63.043 &  $G_9$ & 64.981   & $G_{16}$ & 59.396  \\
\hline $G_3$ & 62.880 &  $G_{10}$  & 61.399  & $G_{17}$ & 67.882   \\
\hline $G_4$ & 57.336 &  $G_{11}$ & 62.375  & $G_{18}$ &  57.517 \\
\hline $G_5$ & 60.638 &  $G_{12}$ & 67.882 & $G_{19}$  & 66.096  \\
\hline $G_6$ & 63.403 &  $G_{13}$  & 61.000 & $G_{20}$ &  59.396 \\
\hline $G_7$ & 63.969 &  $G_{14}$  & 65.835 & $G_{21}$  & 50.911  \\
\hline
\end{tabular}
\end{footnotesize}
\end{center}
\begin{center}
{Table 2.} Sombor energy of cubic graphs of order $10$.
\end{center}

\begin{proposition}\label{prop-unique}
Six cubic graphs of order $10$   are not  ${\cal EN_{SO}}$-unique.
\end{proposition}

\begin{proof}
By observing Table 2, we see that $[G_1]=\{G_1,G_8\}$, $[G_{12}]=\{G_{12},G_{17}\}$ and $[G_{16}]=\{G_{16},G_{20}\}$. Therefore we have fifteen cubic graphs of order $10$  which are ${\cal EN_{SO}}$-unique.\quad\qed
\end{proof}

As an immediate result of Proposition \ref{prop-unique}, we have:

\begin{corollary}
In general, two $k$-regular graphs of the same order may not have same Sombor energy.
\end{corollary}

\begin{figure}[!h]
	\begin{center}
		\psscalebox{0.55 0.55}
{
\begin{pspicture}(0,-5.785)(9.63,3.085)
\psdots[linecolor=black, dotsize=0.4](4.8,2.435)
\psdots[linecolor=black, dotsize=0.4](4.8,0.435)
\psdots[linecolor=black, dotsize=0.4](2.8,-1.165)
\psdots[linecolor=black, dotsize=0.4](3.6,-3.565)
\psdots[linecolor=black, dotsize=0.4](6.0,-3.565)
\psdots[linecolor=black, dotsize=0.4](6.8,-1.165)
\psline[linecolor=black, linewidth=0.08](4.8,0.435)(3.6,-3.565)(6.8,-1.165)(2.8,-1.165)(6.0,-3.565)(4.8,0.435)(4.8,0.435)
\psdots[linecolor=black, dotsize=0.4](8.8,-0.365)
\psdots[linecolor=black, dotsize=0.4](0.8,-0.365)
\psdots[linecolor=black, dotsize=0.4](6.8,-5.165)
\psdots[linecolor=black, dotsize=0.4](2.8,-5.165)
\psline[linecolor=black, linewidth=0.08](4.8,2.435)(8.8,-0.365)(6.8,-5.165)(2.8,-5.165)(0.8,-0.365)(4.8,2.435)(4.8,0.435)(4.8,0.435)
\psline[linecolor=black, linewidth=0.08](6.8,-1.165)(8.8,-0.365)(8.8,-0.365)
\psline[linecolor=black, linewidth=0.08](6.0,-3.565)(6.8,-5.165)(6.8,-5.165)
\psline[linecolor=black, linewidth=0.08](3.6,-3.565)(2.8,-5.165)(2.8,-5.165)
\psline[linecolor=black, linewidth=0.08](2.8,-1.165)(0.8,-0.365)(0.8,-0.365)
\rput[bl](4.68,2.835){$v_1$}
\rput[bl](0.0,-0.365){$v_2$}
\rput[bl](2.18,-5.785){$v_3$}
\rput[bl](6.82,-5.785){$v_4$}
\rput[bl](9.26,-0.445){$v_5$}
\rput[bl](5.02,0.635){$v_6$}
\rput[bl](2.72,-0.765){$v_7$}
\rput[bl](2.98,-3.565){$v_8$}
\rput[bl](6.24,-3.605){$v_9$}
\rput[bl](6.56,-0.725){$v_{10}$}
\end{pspicture}
}
	\end{center}
	\caption{Petersen graph } \label{petersen}
\end{figure}
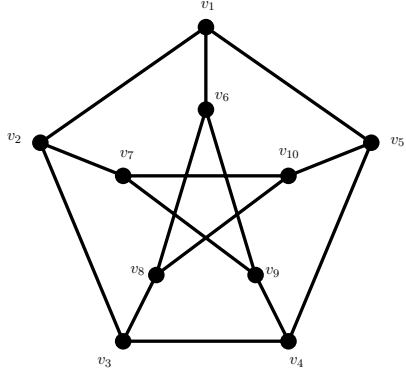

\begin{theorem}
Let ${\cal G}$ be the family of $3$-regular graphs of order $10$. For the Petersen graph $P$ (Figure \ref{petersen}), we have the following properties:
\begin{itemize}
\item[(i)]
$P$  is not ${\cal EN_{SO}}$-unique in ${\cal G}$.
\item[(ii)]
$P$ has the maximum sombor energy in ${\cal G}$.
\end{itemize}
\end{theorem}

\begin{proof}
\begin{itemize}
\item[(i)]
The Sombor matrix of $P$ is
$$A_{SO}(P)=\left( \begin{array}{cccccccccc}
0&3\sqrt{2} &0 &0 &3\sqrt{2} &3\sqrt{2} &0 & 0&0 & 0 \\
3\sqrt{2}&0 &3\sqrt{2} &0 &0 &0 &3\sqrt{2} & 0&0 & 0 \\
0&3\sqrt{2} &0 &3\sqrt{2} &0 &0 &0 & 3\sqrt{2}&0 & 0 \\
0&0 &3\sqrt{2} &0 &3\sqrt{2} &0 &0 & 0&3\sqrt{2} & 0 \\
3\sqrt{2}&0 &0 &3\sqrt{2} &0 &0 &0 & 0&0 & 3\sqrt{2} \\
3\sqrt{2}&0 &0 &0 &0 &0 &0 & 3\sqrt{2}&3\sqrt{2} & 0 \\
0&3\sqrt{2} &0 &0 &0 &0 &0 & 0&3\sqrt{2} & 3\sqrt{2} \\
0&0 &3\sqrt{2} &0 &0 &3\sqrt{2} &0 & 0&0 & 3\sqrt{2} \\
0&0 &0 &3\sqrt{2} &0 &3\sqrt{2} &3\sqrt{2} & 0&0 & 0 \\
0&0 &0 &0 &3\sqrt{2} &0 &3\sqrt{2} & 3\sqrt{2}&0 & 0 \\
\end{array} \right).$$
So
\begin{align*}
\phi _{SO}(P,\lambda)&=det(\lambda I -A_{SO}(P))=(\lambda -9\sqrt{2})(\lambda +6\sqrt{2})^4 (\lambda-3\sqrt{2})^5.
\end{align*}
Therefore we have:
$$\lambda _1=9\sqrt{2}~~,~~\lambda _2=\lambda _3=\lambda _4=\lambda _5=-6\sqrt{2}~~,~~\lambda _6=\lambda _7=\lambda _8=\lambda _9=\lambda _{10}=3\sqrt{2},$$
and so we have $En_{SO}(P)=48\sqrt{2}$. By Table 2, we have $P\in \{G_{12},G_{17}\}$. Hence $P$  is not ${\cal EN_{SO}}$-unique  in ${\cal G}$.
\item[(ii)]
It follows from Part (i) and Table 2.\quad\qed
\end{itemize}
\end{proof}

Now we check that is there any relationship between Sombor energy and permanent of adjacency matrix of two connected $k$-regular graphs of the same order?

\begin{theorem}\label{permanent}
If two connected $k$-regular graphs have the same Sombor energy, then their adjacency matrices may have or have not the same permanent.
\end{theorem}

\begin{proof}
We consider to the cubic graphs of order 10. By Table 2, $ En_{SO}(G_1)=En_{SO}(G_8)$ and $ En_{SO}(G_{16})=En_{SO}(G_{20})$. Now we find $per(A(G_1))$, $per(A(G_8))$, $per(A(G_{16}))$ and $per(A(G_{20}))$. For graph $G_1$, we have
$$A(G_1)=\left( \begin{array}{cccccccccc}
0&1 &0 &0 &0 &1 &0 & 0&0 & 1 \\
1&0 & 1& 1&0 &0 &0 &0 &0 &0  \\
0&1& 0& 1&1 &0 &0 &0 &0 &0  \\
0&1 & 1& 0&1 &0 &0 &0 &0 &0  \\
0&0 & 1& 1&0 &1 &0 &0 &0 &0  \\
1&0 & 0& 0&1 &0 &1 &0 &0 &0  \\
0&0 & 0& 0&0 &1 &0 &1 &1 &0  \\
0&0 & 0& 0&0 &0 &1 &0 &1 &1  \\
0&0 & 0& 0&0 &0 &1 &1 &0 &1  \\
1&0 & 0& 0&0 &0 &0 &1 &1 &0  \\
\end{array} \right).$$

By Ryser's method, we have $per(A(G_1))=72$. Similarly we have:
$$per(A(G_8))=72 ~~,~~ per(A(G_{16}))=144 ~~,~~ per(A(G_{20}))=180.$$
So we have the result. \quad\qed
\end{proof}

By Theorem \ref{permanent}, we know that If two connected $k$-regular graphs have the same Sombor energy, we can say nothing about the permanent of their adjacency matrices. Now by the following Remark, we show that if two graphs have the same permanent, then we can not conclude that they have same Sombor energy. Therefore, in general, there is no relation between Sombor energy and permanent of adjacency matrices of $k$-regular graphs with the same order.

\begin{remark}\label{remark-per}
In the class of cubic graphs of order 10, We have $per(A(G_7))=per(A(G_{11}))=85$, but as we see in Table 2, $En_{SO}(G_7)\neq En_{SO}(G_{11})$.
\end{remark}

As an observation, we see that every graph does not have integer-valued Sombor energy. We end this section with the following conjecture:

\begin{conjecture}\label{conj3}
There is no graph with integer-valued Sombor energy.
\end{conjecture}

\section{Conclusions}

In this paper, we obtained the Sombor characteristic polynomial and Sombor energy of specific graphs such as paths, cycles, stars, complete bipartite graphs and complete graphs. Also we studied Sombor energy of 2-regular and 3-regular graphs. 

Future topics of interest for future research include the following suggestions:

\begin{itemize}
\item[•]
Proving Conjecture \ref{conj3} or Giving a graph with integer-valued Sombor energy.
\item[•]
What is the relationship between $En_{SO}(G)$ and $En_{SO}(G-e)$ where $e\in E(G)$?
\item[•]
What can we say about $En_{SO}(G)$ and $En_{SO}(G-v)$ where $v\in V(G)$?
\item[•]
What is the sombor energy of $G*H$ where $*$ is some kind of operation on two graph?
\item[•]
If two graphs of the same order be {\it ${\cal EN_{SO}}$-equivalent}, do they have any  properties in common?
\end{itemize}

	\section{Acknowledgements} 
	
	The  author would like to thank the Research Council of Norway and Department of Informatics, University of
	Bergen for their support.

\end{document}